\documentclass{article}
\usepackage{geometry}
\usepackage{amsmath}
\usepackage{amsfonts}
\usepackage{amssymb}
\usepackage{amsthm}
\usepackage{enumerate}
\usepackage{hyperref}

\newtheorem{theorem}{Theorem}
\newtheorem{definition}[theorem]{Definition}
\newtheorem{proposition}[theorem]{Proposition}

\newcommand{\R}{\mathbb{R}}

\newcommand{\norm}[1]{\left\lVert#1\right\rVert}

\newcommand{\UD}{\operatorname{UD}}

\renewcommand{\epsilon}{\varepsilon}

\title{Bounds on hyperbolic sphere packings: On a conjecture by Cohn and Zhao}
\author{Maximilian Wackenhuth}
\date{}
\begin{document}

% Use the \maketitle command after the abstract
\maketitle
\abstract{We prove sphere packing density bounds in hyperbolic space (and more generally irreducible symmetric spaces of noncompact type), which were conjectured by Cohn and Zhao and generalize Euclidean bounds by Cohn and Elkies. We work within the Bowen-Radin framework of packing density and replace the use of the Poisson summation formula in the proof of the Euclidean bound by Cohn and Elkies with an analogous formula arising from methods used in the theory of mathematical quasicrystals.}

% Example of section
		\section{Introduction}
Let $(X,d)$ be a metric space equipped with a Borel measure $\mathrm{vol}$. An $r$-sphere packing $P$ in $X$ is a collection of disjoint open balls of radius $r$. This article is concerned with the density $D(P)$ of ``nice'' sphere packings, where the definition of ``nice'' depends on the space $X$ under consideration. For a ``nice'' sphere packing $P$ the density $D(P)$ is given by
\begin{equation}\label{NaiveDensity}
	D(P) = \lim_{R\to \infty}\frac{\mathrm{vol}(B_R(x)\cap \bigcup P)}{\mathrm{vol}(B_R(x))},\end{equation}
where $x\in X$ is an auxiliary point. In Euclidean space, a sphere packing is called \emph{nice}, if the limit above exists for some point $x$, in which case it is independent of $x$. 

The generally accepted notion of \emph{nice} sphere packing in hyperbolics spaces was introduced by Bowen and Radin in \cite{Bowen2003} and \cite{Bowen2004}. It is based on ergodic theoretic considerations and excludes certain pathological examples (see \cite{Boroczky}). We will review the necessary definitions in Section \ref{DensityFramework}. For this introduction it suffices to know that for a nice hyperbolic sphere packing the density $D(P)$ is defined and independent of $x$. The following three problems have guided much of the research on sphere packing densities in Euclidean and hyperbolic spaces:
\begin{enumerate}[(i)]
	\item What is the maximal packing density $\triangle_{\R^n}(r):= \sup \{D(P)\mid P \text{ nice $r$-sphere packing}\}$?
	\item Does there always exists a nice $r$-sphere packing $P$ with $D(P) = \triangle_{\R^n}(r)$?
	\item Can we construct an explicit nice $r$-sphere packing $P$ with $D(P) = \triangle_{\R^n}(r)$?
\end{enumerate}
\subsection{Euclidean packings}
In Euclidean space Problem (ii) can be answered affirmatively, see for instance \cite{GroemerEuclidean}. Answers to Problems (i) and (iii) are known in dimension $2$ by the work of Fejes Tóth (see \cite{Fejes}), in dimension $3$ by the work of Hales (see \cite{Hales}), in dimension $8$ by the work of Viazovska (see \cite{ViazovskaDim8}) and in dimension $24$ by the work of Cohn, Kumar, Miller, Radchenko and Viazovska (see \cite{ViazovskaDim24}). The proofs of the results in dimensions $8$ and $24$ are based on the following upper bound on $\triangle_{\R^n}(r)$ obtained by Cohn and Elkies in \cite{CohnElkies2003}. 
\begin{theorem}[Cohn--Elkies]\label{euclideanLPbounds}
	Assume that $f:\R^n\to \R$ is continuous and rapidly decaying\footnote{For precise conditions see \cite{CohnElkies2003}.} and satisfies
	\begin{enumerate}[(i)]
		\item $f(x)\leq 0$ for $\norm{x}\geq 2r$,
		\item $\widehat{f}\geq 0$ and $\widehat{f}(0) > 0$, where $\widehat{f}$ is the Fourier transform of $f$.
	\end{enumerate}
	Then $\triangle_{\R^n}(r)\leq \mathrm{vol}(B_{r})\frac{f(0)}{\widehat{f}(0)}$.
\end{theorem}
Using Poisson summation they show the estimate above for periodic sphere packings. The general statement then follows from the classical fact that $\triangle_{\R^n}(r)$ can be approximated by densities of periodic sphere packings.

\subsection{Hyperbolic packings}
In hyperbolic space, Problem (ii) was answered affirmatively by Bowen and Radin.  Answers to Problems (i) and (iii) are known only in low dimensions for specific radii \cite{Bowen2003}. It is a notorious open problem whether the optimal packing density $\triangle_{\mathbb{H}^n}(r)$ can be approximated by densities of periodic packings for all $n$ and $r$. For $n=2$ (and arbitrary $r$) this was answered in the positive by Bowen \cite{BowenPeriodicityDim2}, who conjectured that his proof should generalize also to $n=3$, but this has not been carried out yet. In higher dimensions his methods do not apply. If one asks for the stronger property that $\triangle_{\mathbb{H}^n}(r)$ be realized by a periodic packing, then the answer becomes negative in general  \cite{Bowen2003}:		
\begin{theorem}[Bowen--Radin]\label{PeriodicityAndNonperiodicity}
	If $P$ is a nice $r$-sphere packing in $\mathbb{H}^n$ with $D(P) = \triangle_{\mathbb{H}^n}(r)$, then $P$ is not periodic, with the exception of at most countably many $r\geq 0$.
\end{theorem}		
If one restricts attention to periodic packings, then Theorem \ref{euclideanLPbounds} can be generalized to the hyperbolic case. Indeed, the following theorem was obtained by Cohn and Zhao \cite{CohnZhao} based on unpublished work by Cohn, Lurie and Sarnak.
\begin{theorem}[Cohn--Lurie--Sarnak, Cohn--Zhao]\label{PeriodicHyperbolicEstimate}
	Choose a base point $o\in \mathbb{H}^n$ and assume that $f:\mathbb{H}^n\to \R$ is radial, continuous and rapidly decaying\footnote{For precise conditions see \cite{CohnZhao}.} and satisfies 
	\begin{enumerate}[(i)]
		\item $f(x)\leq 0$ for $d(x,o)\geq 2r$,
		\item $\widehat{f}\geq 0$ and $\widehat{f}(\mathbf{1}) > 0$, where $\widehat{f}$ is the spherical transform of $f$.
	\end{enumerate}
	Then $D(P)\leq \mathrm{vol}(B_{r})\frac{f(o)}{\widehat{f}(\mathbf{1})}$ for any periodic $r$-sphere packing $P$ of \space $\mathbb{H}^n$.
\end{theorem}
The proof is similar to that of Theorem \ref{euclideanLPbounds}, but replaces Poisson summation by a pre-trace formula. In dimension $2$, Theorem \ref{PeriodicHyperbolicEstimate} provides a bound on $\triangle_{\mathbb{H}^2}(r)$ for every $r$ via periodic approximation. Cohn and Zhao conjectured in \cite{CohnZhao} that the analogous bound should also hold in higher dimension. They also showed that this conjecture implies bounds on $\triangle_{\mathbb{H}^n}(r)$ which beat the hyperbolic Levenshtein-Kapatiansky bounds (see \cite{Levenstein}), just as in the Euclidean case. The purpose of this note is to announce a proof of the following version\footnote{Our regularity assumptions are slightly stronger than those conjectured by Cohn and Zhao.} of the Cohn-Zhao conjecture \cite{Wackenhuth}.
\begin{theorem}\label{hyperbolicLPbound}
	Let $S$ be an irreducible symmetric space of noncompact type (e.g. $S=\mathbb{H}^n$) and assume that $f:S\to \R$ is a radial, continuous and rapidly decaying function\footnote{The precise condition is that $f$ is a Harish-Chandra $L^1$-Schwartz function.} satisfying
	\begin{enumerate}[(i)]
		\item $f(x)\leq 0$ for $d(x, o)\geq 2r$,
		\item $\widehat{f}\geq 0$ and $\widehat{f}(\mathbf{1})>0$, where $\widehat{f}$ denotes the spherical transform of $f$.
	\end{enumerate}
	Then $\triangle_{\mathbb{H}^n}(r)\leq \mathrm{vol}(B_r)\frac{f(o)}{\widehat{f}(\mathbf{1})}$.
\end{theorem}
Due to the unclear status of the periodic approximation problem we were naturally guided towards methods inspired by aperiodic order. Our proof of Theorem \ref{hyperbolicLPbound} follows the proofs of Theorem \ref{euclideanLPbounds} and Theorem \ref{PeriodicHyperbolicEstimate}, but instead of Poisson summation or a pre-trace formula we use a weak aperiodic analogue of a summation formula, inspired by methods from the field of mathematical quasicrystals (particularly methods developed in \cite{BHP3}). Note that recent results concerning \emph{lower} bounds on $\triangle_{\R^n}(r)$ and $\triangle_{\mathbb{H}^n}(r)$ also use aperiodic methods, see the work of  Campos, Jenssen, Michelen and Sahasrabudhe \cite{EuclideanLowerArXiv} and Fernández, Kim, Liu and Pikhurko \cite{HyperbolicLowerJournal}  respectively.

\subsection{Disclosure statement}

The authors do not work for, consult, own shares in or receive funding from any
company or organization that would benefit from this article, and have disclosed
no relevant affiliations beyond their academic appointment.

\section{Bowen-Radin packing density framework}\label{DensityFramework}
Before we sketch the proof of our main result, we first give a quick overview of the Bowen-Radin framework of packing density in the general case of irreducible symmetric spaces of noncompact type.
All results in this section were developed by Bowen and Radin in \cite{Bowen2003} and \cite{Bowen2004} for hyperbolic space and generalize directly to the case of an irreducible symmetric space of noncompact type \cite{HartnickWackenhuthDensity}.

Assume from now on that $G$ is a simple noncompact Lie group with finite center, fix a maximal compact subgroup $K\subset G$ and set $o:= eK\in S:= G/K$. Let $d$ denote the length-metric on the symmetric space $G/K$ and $\pi:G\to G/K$ the quotient map. Let $m_G$ denote the Haar measure on $G$, $m_K$ the Haar measure on $K$, normalized such that $m_K(K) = 1$, and $m_S$ the $G$-invariant Borel measure on $G/K$ such that the Weil desintegration formula holds for $m_G, m_K, m_S$. Note that $m_S$ is equal to the Riemannian volume measure $\mathrm{vol}$ which we used in the statement of Theorem \ref{hyperbolicLPbound}. 

We say that $A\subset G/K$ is $r$-uniformly discrete if $d(x,y)\geq 2r$ for all $x\neq y\in A$ and denote by $\UD_r(S)$ the set of all $r$-uniformly discrete subsets of $S$. Note that each element $P\in\UD_r(S)$ corresponds to the set of ball centers of an $r$-sphere packing, denoted $P^r$. We equip $\UD_r(S)$ with the Chabauty-Fell topology (see e.g. \cite{BHP1}). With this topology $\UD_r(S)$ is a compact Hausdorff space and the natural $G$-action on $\UD_r(S)$ is continuous. 
\begin{definition}
	An \emph{invariant random $r$-sphere packing} (in short $r$-IRP) is a $G$-invariant Borel probability measure $\mu$ on $\UD_r(S)$.
\end{definition}
If we set $\Omega := \UD_r(S)$ and $\mathbb{P} := \mu$, we can interpret an $r$-IRP as the distribution of an $\UD_r(S)$-valued random variable (or a ``hard sphere point process'' in probabilistic language) \[\Lambda_\mu: (\Omega, \mathbb{P})\to \UD_r(S),\ P\mapsto P.\]	
The packing density of such a point process is defined as 	
\[D(\Lambda_\mu) := \mathbb{P}(\#(\Lambda_\mu \cap B_r(o)) > 0) = \mathbb{E}[\#(\Lambda_\mu \cap B_r(o))].\]
We now turn to deterministic packings. Informally speaking, a nice packing is a generic instance of an ergodic IRP; the actual definition is as follows:
\begin{definition}
	Let $P\in \UD_r(S)$. The $r$-sphere packing $P^r$ in $S$ is called \emph{nice}, if there is an ergodic $r$-IRP $\mu$ such that
	\[\frac{1}{m_S(B_R(o))}\int_{\pi^{-1}(B_R(o))}g(h^{-1}P')dm_G(h)\to \int g(Q)d\mu(Q), \quad\text{as}\quad R\to\infty\] holds for the function $g:\UD_r(S)\to \R,\ \tilde{P}\mapsto \#(\tilde{P}\cap B_r(o))$ and every element $P'$ in the $G$-orbit of $P$ in $\UD_r(S)$.
\end{definition}
\begin{definition}\label{triangleProb}
	For all $r>0$ we define the \emph{optimal packing density} $\triangle_S(r)$ of $S$ as
	\[\triangle_S(r) := \sup\{D(P^r)\mid P^r\text{ is nice }\}.\]
\end{definition}  
The  pointwise ergodic theorem for semisimple Lie groups by Gorodnik and Nevo in \cite{GorodnikNevoLatticeSubgroups} implies that for any ergodic $r$-IRP $\mu$ and $\mu$-almost every $P\in \UD_r(S)$ the $r$-packing $P^r$ is nice. Moreover, for any such packing $P^r$ the limit defining $D(P^r)$ exists and coincides with $D(\Lambda_\mu)$, hence in particular is independent of the basepoint $x$ used to define it. In conjunction with the ergodic decomposition theorem one obtains
	$ \triangle_{S}(r)=\sup\{D(\Lambda_\mu)\mid \mu \text{ is an $r$-IRP }\}$ and thus $\triangle_S(r)$ as defined above agrees with the optimal density as defined by Bowen and Radin in the case $S=\mathbb{H}^n$.
%Thus, in order to bound $\triangle_n(r)$, it suffices to bound $D(\Lambda_\mu)$ for every $r$-IRP $\mu$.
\section{Replacing Poisson summation}
Our aim for this section will be a sketch of a proof of our main theorem.
For the proof we introduce the autocorrelation and reduced autocorrelation measure of an $r$-IRP, the main ingredients of our replacement for the Poisson summation formula. We will relate these to a number called the intensity of an $r$-IRP and use this to construct an upper bound on the packing density of an $r$-IRP.

\begin{definition}[Björklund--Byléhn, \cite{MichaelMattiasArXiv}]
	Let $\mu$ be an $r$-IRP. The \emph{autocorrelation measure} of $\mu$ is defined by
	\[\eta_\mu^+(f):= \int_{\UD_r(S)}\sum_{y\in P}\sum_{x\in P}f(\sigma(x)^{-1}\sigma(y))b(x)d\mu(x)\quad \text{for }f\in C_c(G),\]
	where $\sigma:S\to G$ is a Borel section and  $b: S\to \R$ is measurable, bounded with bounded support and $m_S(b)=1$.
\end{definition}
The proof of Lemma 2.8 in \cite{MichaelMattiasArXiv} shows that $\eta_\mu$ does not depend on our choice of $b$ or $\sigma$.
As $\mu$ is $G$-invariant, the map $B\mapsto \mathbb{E}[\#(\Lambda_\mu\cap B)]$ defines a $G$-invariant Borel measure and thus there is a constant $i_\mu\geq 0$ with $\mathbb{E}[\#(\Lambda_\mu\cap B)] = i_\mu m_S(B)$ for every Borel set $B\subset S$. The constant $i_\mu$ is called the \emph{intensity} of $\mu$. Note that \[D(\Lambda_\mu) = i_\mu m_S(B_r(o)).\] Thus finding an upper bound on $i_\mu$ will yield an upper bound on $D(\Lambda_\mu)$.

\begin{definition}
	The reduced autocorrelation $\eta_\mu$ of an $r$-IRP $\mu$ is defined by 
	\[\eta_\mu(f) := \eta_\mu^+(f) - i_\mu^2 m_G(f)\quad\text{for }f\in C_c(G).\]
\end{definition}
Lemma 2.8 in \cite{MichaelMattiasArXiv} now implies:
\begin{proposition}
	$\eta_\mu$ and  $\eta_\mu^+$ are positive-definite Borel measures on $G$. 
\end{proposition}
\begin{proof}[Proof sketch for Theorem \ref{hyperbolicLPbound}]
	Assume that $F:G\to \R$ is a nice function which lifts a function $f:S\to \R$ satisfying the conditions of the theorem. Now let $\mu$ be an arbitrary $r$-IRP.
	As $\eta_\mu$ and $\eta_\mu^+$ are positive-definite, their spherical transforms are positive Borel measures on the spherical spectrum of $(G, K)$.
	Then the Plancherel theorem implies that 
	\begin{align}\label{e:estimate_1}
	\eta_\mu^+(F) = \widehat{\eta}_\mu^+(\widehat{F}) = \widehat{\eta}_\mu(\widehat{F}) + i_\mu^2\widehat{m_G}(\widehat{F}) \geq i_\mu^2\widehat{F}(\mathbf{1}),
	\end{align}
	where we have used (ii), the fact that $\widehat{\eta}_\mu$ is a positive measure and $\widehat{m_G} = \delta_{\mathbf{1}}$.
	We also have 
	\[
	\eta_\mu^+(F) = \int_{\UD_r(S)}\sum_{y\in P}\sum_{x\in P}F(\sigma(x)^{-1}\sigma(y))b(x)d\mu(x).
	\]
	Fix  $T>0$ such that $m_S(B_T(o)) = 1$ and set $b=\chi_{B_T(o)}$.
	Then
	\begin{align*}
		\eta_\mu^+(F) &= \int_{\UD_r(S)}\sum_{y\in P}\sum_{x\in P\cap B_T(o)}F(\sigma(x)^{-1}\sigma(y))d\mu(x)\\
		&\leq \int_{\UD_r(S)}\sum_{x\in P\cap B_T(o)}F(\sigma(x)^{-1}\sigma(x))d\mu(x)\\
		%&= \int_{\UD_r(S)}f(e)\sum_{x\in P\cap B_T(o)}d\mu(P)\\
		%&= f(e)\int_{\UD_r(S)}\#(P\cap B_T(o))d\mu(P)\\
		&= F(e)\mathbb{E}[\#(\Lambda_\mu\cap B_T(o))]\\
		&= F(e)i_\mu m_S(B_T(o)) = F(e)i_\mu,
	\end{align*}
	using property (i) for the inequality. Joining it with equation \eqref{e:estimate_1}, we obtain \[i_\mu \leq \frac{F(e)}{\widehat{F}(\mathbf{1})}=\frac{f(o)}{\widehat{f}(\mathbf{1})} \implies D(\Lambda_\mu) \leq m_S(B_r)\frac{f(o)}{\widehat{f}(\mathbf{1})}.\]
	Passing to the supremum over all $\mu$ then yields the desired bound by Theorem \ref{triangleProb}.
	
	The proof above can for example be performed if $F$ is the convolution of two bi-$K$-invariant compactly supported continuous functions. The extension to Harish-Chandra $L^1$-Schwartz functions is then achieved by approximation arguments.
\end{proof}
% The next command determines the bibliography style. Please do not
% change this.
\bibliographystyle{plain}
%  This inserts the bib file

\end{document}